\documentclass[11pt]{amsart}
\usepackage[oldstylemath,fulloldstylenums]{kpfonts}

\usepackage{amsmath}
\usepackage{amsxtra}
\usepackage{amsthm}

\usepackage{fullpage}

\usepackage{mathrsfs}
\usepackage{tikz-cd}
\usepackage{mathtools}

\usepackage[bbgreekl]{mathbbol}
\usepackage{stmaryrd}
\usepackage{enumitem} 
\usepackage{xcolor}
\definecolor{darkgreen}{rgb}{0,0.45,0} 
\definecolor{darkred}{rgb}{0.75,0,0}
\definecolor{darkblue}{rgb}{0,0,0.6} 
\usepackage[hyphens]{url}
\usepackage[pagebackref,colorlinks,citecolor=darkgreen,linkcolor=darkgreen,urlcolor=darkblue]{hyperref}
\usepackage{comment}

\newtheorem{thm}{Theorem}[section]
\newtheorem{lem}[thm]{Lemma}
\newtheorem{prop}[thm]{Proposition} 
\newtheorem{cor}[thm]{Corollary}
\newtheorem*{thm*}{Theorem}

\theoremstyle{definition}
\newtheorem{defn}[thm]{Definition}

\theoremstyle{remark}

\makeatletter
\let\c@equation\c@thm
\makeatother
\numberwithin{equation}{section}

\newcommand{\refl}{\textup{refl}}
\newcommand{\Map}{\textup{Map}}
\newcommand{\Comp}{\textup{Comp}}
\newcommand{\res}{\textup{res}}
\newcommand{\id}{\textup{id}}
\renewcommand{\lim}{\textup{lim}}
\newcommand{\colim}{\textup{colim}}
\newcommand{\dom}{\textup{dom}}
\newcommand{\cod}{\textup{cod}}

\newcommand{\To}{\Rightarrow}
\newcommand{\Hom}{\textup{Hom}}

\newcommand{\Group}{\textup{Group}}

\newcommand{\NN}{\mathbb{N}}

\newcommand{\RR}{\mathbb{R}}

\newcommand{\ua}{\textup{ua}}
\newcommand{\cU}{\mathcal{U}}

\newcommand{\CAT}{\mathsf{CAT}}
\newcommand{\ooCAT}{\infty\text{-}\mathsf{CAT}}

\DeclareMathAlphabet{\mathbbe}{U}{bbold}{m}{n}

\newcommand{\2}{\mathbbe{2}}

\begin{document}

\title{Could \texorpdfstring{$\infty$}{infinity}-category theory be taught to undergraduates?}

    \author{Emily Riehl}
\date{\today}

\thanks{This is the author's version of
an article from the \emph{Notices of the AMS} with an expanded reference list. Tim Campion, Maru Sarazola, Mike Shulman, Dominic Verity, and Jonathan Weinberger all provided helpful suggestions on an early draft version of this article, while John Bourke suggested the rebranding of absolute lifting diagrams as relative adjunctions. Further improvements were made in response to comments from Steven Sam and three anonymous referees. The author is also grateful to receive support from the NSF via the grants DMS-1652600 and DMS-2204304, from the ARO under MURI Grant W911NF-20-1-0082, from the AFOSR under award number FA9550-21-1-0009, and from a Simons Foundation Fellowship.}

\address{Dept.~of Mathematics\\Johns Hopkins University \\ 3400 N Charles Street \\ Baltimore, MD 21218}
\email{eriehl@jhu.edu}

\begin{abstract}
    The extension of ordinary category theory to $\infty$-categories at the start of the 21st century was a spectacular achievement pioneered by Joyal and Lurie with contributions from many others. Unfortunately, the technical arguments required to solve the infinite homotopy coherence problems inherent in these results make this theory difficult for non-experts to learn. This essay surveys two programs that seek to narrow the gap between $\infty$-category theory and ordinary 1-category theory. The first leverages similarities between the categories in which 1-categories and $\infty$-categories live as objects to provide ``formal'' proofs of standard categorical theorems. The second, which is considerably more speculative, explores $\infty$-categories from new ``univalent'' foundations closely related to homotopy type theory.
\end{abstract}

\maketitle

\setcounter{tocdepth}{1}
\tableofcontents

\section{The algebra of paths}

It is natural to probe a suitably-nice topological space $X$ by means of its \emph{paths}, the continuous functions from the standard unit interval $I\coloneqq [0,1] \subset \RR$ to $X$. But what structure do the paths in $X$ form?

To start, the paths form the edges of a directed graph whose vertices are the points of $X$: a path $p \colon I \to X$ defines an arrow from the point $p(0)$ to the point $p(1)$. Moreover, this graph is \emph{reflexive}, with the constant path $\refl_x$ at each point $x \in X$ defining a distinguished endoarrow.

Can this reflexive directed graph be given the structure of a category? To do so, it is natural to define the composite of a path $p$ from $x$ to $y$ and a path $q$ from $y$ to $z$ by gluing together these continuous maps---i.e., by concatenating the paths---and then by reparametrizing via the homeomorphism $I \cong I \cup_{1=0} I$ that traverses each path at double speed:
\begin{equation}\label{eq:path-comp}
    \begin{tikzcd}
        I \arrow[r, "\cong"] \arrow[rr, bend right, "p \ast q"'] & I \cup_{1=0} I \arrow[r, "p \cup q"] & X
    \end{tikzcd}
\end{equation}
But the composition operation $\ast$ fails to be associative or unital. In general, given a path $r$ from $z$ to $w$, the composites $(p \ast q) \ast r$ and $p \ast (q \ast r)$ are not equal: while the have the same image in $X$, their parametrizations differ. However, they are \emph{based homotopic}, in the sense that there exists a continuous function $h \colon I \times I \to X$ so that
\[ h(-,0) \coloneqq (p \ast q) \ast r, \quad h(-,1) \coloneqq p \ast (q \ast r), \quad h(0,-) \coloneqq \refl_x, \quad \text{and} \quad h(1,-) \coloneqq \refl_w,\]
a situation we summarize by writing $(p \ast q) \ast r \simeq p \ast (q \ast r)$. Similarly, the paths $p \ast \refl_y \simeq \refl_x \ast p \simeq p$ are all based homotopic, though not equal (unless $p$ is the constant path).

Paths are also invertible up to based homotopy: for any path $p \colon I \to X$ from $x$ to $y$, its reversal $p^{-1} \colon I \to X$, defined by precomposing with the flipping automorphism $I \cong I$, defines an inverse up to based homotopy: $p \ast p^{-1} \simeq \refl_x$ and $p^{-1}\ast p \simeq \refl_y$.  These observations motivate the following definition:

\begin{defn} For a space $X$, the \textbf{fundamental groupoid} $\pi_1X$ is the category whose:
    \begin{itemize}
        \item objects are the points of $X$ and 
        \item arrows are based homotopy classes of paths of $X$
    \end{itemize}
    with composition defined by concatenation, identity arrows defined by the constant paths, and inverses defined by reversing paths.
\end{defn}

The fundamental groupoid of a space $X$ answers a slightly different question than originally posed, describing the structure formed by the \emph{based homotopy classes of paths} in $X$. The paths themselves form something like a weak groupoid where composition is not uniquely defined. Indeed, given paths $p$ from $x$ to $y$ and $q$ from $y$ to $z$, the composite in $\pi_1X$ is represented by \emph{any} path $s$ so that there is a based homotopy $h$ witnessing $p \ast q \simeq s$. Here, the based homotopy $h$ defines a continuous function from the solid 2-simplex into $X$ that restricts along the boundary triangle to the map defined by gluing the three paths:
\[
    \begin{tikzcd}
        \partial\Delta^2 \arrow[r, "{(p,q,s)}"] \arrow[d, hook] & X \\ \Delta^2 \arrow[ur,  "h"']
    \end{tikzcd}
\] 

While multivalued composition operations are not well-behaved in general, this one has a certain homotopical uniqueness property: the witnessing homotopies for two such composites can be glued together---by a higher dimensional analogue of the construction \eqref{eq:path-comp}---to define a homotopy $s \simeq p \ast q \simeq t$.  This suggests that the witnessing homotopies that fill the triangles formed by the paths between $x$ and $y$ and $z$ might be productively regarded as part of the composition data. Thus, the moduli space of composites of $p$ and $q$ is defined as follows:

\begin{defn}\label{defn:space-of-composites} Given composable paths $p$ and $q$ in a  space $X$, the \textbf{space of composites} of $p$ and $q$ is the subspace $\Comp(p,q) \subset \Map(\Delta^2,X)$ of continuous maps $h \colon \Delta^2 \to X$ that restrict on the boundary horn to $p \cup q$:
    \begin{equation}\label{eq:comp-space}\begin{tikzcd} \Comp(p,q) \arrow[r, hook] \arrow[d] \arrow[dr, phantom, "\lrcorner" very near start] & \Map(\Delta^2,X) \arrow[d, "\res"] \\ \ast \arrow[r, "{p \cup q}"] & \Map(I\cup_{1=0} I,X)
    \end{tikzcd}
\end{equation}
\end{defn}

The homotopical uniqueness of path composition can be strengthened by the following observation:

\begin{thm}\label{thm:unique-composition} For any composable paths $p$ and $q$ in a space $X$, the space of composites $\Comp(p,q)$ is contractible.
\end{thm}

So what do the paths in $X$ form? We have seen that they form a weak groupoid with a multivalued but homotopically unique composition law. In fact, this weak category is an infinite-dimensional category. The based homotopies $h \colon I \times I \to X$ between parallel paths $s$ and $t$ themselves might be regarded as paths $h \colon I \to \Map(I,X)$ between points $s$ and $t$ in the space $\Map(I,X)$ to which Theorem \ref{thm:unique-composition} equally applies, and these observations extend iteratively to higher-dim\-en\-sion\-al homotopies. The points, paths, and higher paths in a space $X$ assemble into a weak infinite-dimensional category---with interacting weak composition and identities for paths at all levels\footnote{Famously, the definition of a weak three-dimensional category by Gordon, Power, and Street takes six pages to state.}---in which all morphisms are weakly invertible. Such a structure is known as an $\infty$-\emph{groupoid} \cite{lurie-what}.

\subsection*{The homotopy hypothesis}

The \emph{fundamental $\infty$-groupoid} $\pi_\infty X$ of paths in $X$ captures the data of all of the higher homotopy groups of the space, information which is referred to as the \emph{homotopy type} of the space. In a famous letter to Quillen \cite{G-PS}, Grothendieck formulated his \emph{homotopy hypothesis}, positing that the fundamental $\infty$-groupoid construction defines an equivalence between homotopy types and $\infty$-groupoids. Grothendieck's vision was that this result should be provable for various models of $\infty$-groupoids that were then under development, though some instead use this thesis to \emph{define} an $\infty$-groupoid to be a homotopy type.

There is something unsatisfying, though, about the na\"{i}ve interpretation of the homotopy hypothesis as the assertion of a bijection between the collection of homotopy types and the collection of $\infty$-groupoids, or even as an equivalence between the homotopy category of spaces\footnote{Homotopy types can be understood as isomorphism classes of objects in the \emph{homotopy category of spaces}, the category obtained by localizing the category of spaces and continuous functions by the \emph{weak homotopy equivalences}, those functions that induce isomorphisms on all homotopy groups.} and the homotopy category of $\infty$-group\-oids. The disappointment lies in the fact that both spaces and $\infty$-groupoids live most naturally as the objects in weak infinite-dimensional categories, with the standard morphisms between them supplemented by higher dimensional weakly invertible morphisms. Thus, a more robust expression of the homotopy hypothesis is the assertion that the $\infty$-\emph{categories} of spaces and of $\infty$-groupoids are equivalent,\footnote{In fact, a version of this result had been proven already by Quillen for the Kan complex  model of $\infty$-groupoids \cite{Quillen}, but this was not so clearly understood at the time.}  so now we must explain what $\infty$-categories are.

\section{\texorpdfstring{$\infty$}{Infinity}-categories in mathematics}

Over the past few decades, \emph{$\infty$-categories}---weak infinite-dimensional categories with weakly invertible morphisms above dimension one---have been invading certain areas of mathematics. In derived algebraic geometry, the derived category of a ring is now understood as the 1-categorical quotient of the $\infty$-category of chain complexes, and an $\infty$-categorical property called ``stability'' explains the triangulated structure borne by the derived category \cite[\S 1.1]{lurie-algebra}. 
 In mathematical physics, Atiyah's topological quantum field theories have been ``extended up'' to define functors between $\infty$-categories \cite{Freed}.\footnote{Lurie's ``fully extended'' topological quantum field theories are also ``extended down'' so that they might be understood as functors between $(\infty,n)$-categories, with non-invertible morphisms up to and including the dimension $n$ of the indexing cobordisms. Here we reserve the term ``$\infty$-categories'' for ``$(\infty,1)$-categories,'' which have non-invertible morphisms only in the bottom dimension \cite{lurie-tqft}.}  
 $\infty$-categories have also made appearances in the Langlands program in representation theory \cite{FS} 
 and in symplectic geometry \cite{NT} 
  among other areas. Quillen's model categories \cite{Quillen} from abstract homotopy theory are now understood as presentations of $\infty$-categories.    

Ordinary categories ``frame a template for a mathematical theory,''  with the objects providing the ``nouns'' and the morphisms the ``verbs,'' in a metaphor suggested by Mazur \cite{Mazur}. 
 As the objects mathematicians study increase in complexity, a more robust linguistic template may be required to adequately describe their natural habitats---with adjectives, adverbs, pronouns, prepositions, conjunctions, interjections, and so on---leading to the idea of an $\infty$-category. Like an ordinary 1-category, an $\infty$-category has objects and morphisms, now thought of as ``1-dimensional'' transformations. The extra linguistic color is provided by higher dimensional invertible morphisms between morphisms---such as chain homotopies or diffeomorphisms---and higher morph\-isms between these morphisms, continuing all the way up.

How might a researcher in one of these areas go about learning this new technology of $\infty$-categories? And how might $\infty$-category theory ultimately be distilled down into something that we could reasonably teach advanced undergraduates of the future?

\subsection*{Curiosities from the literature}

$\infty$-categories were first introduced by Boardman and Vogt to describe the composition structure of homotopy coherent natural transformations between homotopy coherent diagrams \cite{BV}. Joyal was the first to assert that ``most concepts and results of category theory can be extended to [$\infty$-categories]'' and pioneered the development of $\infty$-categorical analogues of standard categorical notions \cite{joyal-quasi,joyal-theory}. 
  Lurie then developed various aspects of $\infty$-category theory that were needed for his thesis on derived algebraic geometry \cite{lurie-on-topoi, lurie-thesis}. 
    His books \cite{lurie-topos,lurie-algebra}
  and the online textbook \emph{Kerodon} \cite{lurie-kerodon} are primary references for uses of this technology, while texts written by Cisinski, Groth, Hinich, Rezk, and others provide parallel introductions to the field \cite{Cisinski, groth, hinich, rezk}.

If one delves further into the $\infty$-categories literature, some curiosities soon become apparent:

\begin{enumerate}
\item Particularly in talks or lecture series introducing the subject, the definition of $\infty$-category is frequently delayed, and when definitions are given, they don't always agree.
\end{enumerate}

 These competing definitions are referred to as \emph{models} of $\infty$-categories, which  are Bourbaki-style mathematical structures defined in terms of sets and functions that represent infinite-dimensional categories with a weak composition law in which all morphisms above dimension one are weakly invertible.  In order of appearance, these include \emph{simplicial categories}, \emph{quasi-categories}, \emph{relative categories}, \emph{Segal categories}, \emph{complete Segal spaces}, and \emph{1-com\-pli\-cial sets}, each of which comes with an associated array of naturally occurring examples \cite{AntolinCamarena:2013aw, Bergner:2018ht}. 
 \begin{enumerate}[resume]
\item Considerable work has gone into defining the key notions for and proving the fundamental results about $\infty$-categories, but sometimes this work is  later redeveloped starting from a different model.
\end{enumerate}
For instance, \cite{KV} begins:
\begin{quote}
In recent years $\infty$-categories or, more formally, $(\infty$,1)-categories appear in various areas of mathematics. For example, they became a necessary ingredient in the geometric Langlands problem. In his books \cite{lurie-topos,lurie-algebra} Lurie developed a theory of $\infty$-cat\-e\-gor\-ies in the language of quasi-categories and extended many results of the ordinary category theory to this setting. 

In his work \cite{rezk-CSS} Rezk introduced another model of $\infty$-cat\-e\-gor\-ies, which he called complete Segal spaces. This model has certain advantages. For example, it has a generalization to $(\infty,n)$-categories (see \cite{rezk-cartesian}).

 It is natural to extend results of the ordinary category theory to the setting of complete Segal spaces. In this note we do this for the Yoneda lemma.\end{quote}
\begin{enumerate}[resume]
\item Alternatively, authors decline to pick a model at all and instead work ``model-in\-de\-pen\-dent\-ly.''
\end{enumerate}
One instance of this appears in the precursor \cite{lurie-on-topoi} to \cite{lurie-topos}, which avoids selecting a model of $\infty$-categories at all:
\begin{quote}
We will begin in \S 1 with an informal review of the theory of $\infty$-categories. There are many approaches to the foundation of this subject, each having its own particular merits and demerits. Rather than single out one of those foundations here, we shall attempt to explain the ideas involved and how to work with them. The hope is that this will render this paper readable to a wider audience, while experts will be able to fill in the details missing from our exposition in whatever framework they happen to prefer.
\end{quote}

The fundamental obstacle to giving a uniform definition of an $\infty$-category is that our traditional set-based foundations for mathematics are not really suitable for reasoning about $\infty$-categories: sets do not feature prominently in $\infty$-categorical data, especially when $\infty$-categories are only well-defined up to equivalence, as they must be when different models are involved. When considered up to equivalence,   $\infty$-categories, like ordinary categories,  do not have a well-defined set of objects.  Essentially, $\infty$-categories are 1-categories in which all the sets have been replaced by $\infty$-groupoids.
Where a category has a set of elements, an $\infty$-category has an $\infty$-groupoid of elements, and where a category has hom-sets of morphisms, $\infty$-categories have $\infty$-groupoidal mapping spaces.\footnote{``Large'' $\infty$-categories also exist and behave like large 1-categories.}  The axioms that turn a directed graph into a category are expressed in the language of set theory: a category has a composition function satisfying axioms expressed in first-order logic with equality. By analogy, composition in an $\infty$-category can be understood as a morphism between $\infty$-groupoids, but such morphisms no longer define functions since homotopy types do not have underlying sets of points.\footnote{Similar considerations have motivated Scholze et al to use the term ``anima''---referring to the ``soul'' of a ``space''---as a synonym for $\infty$-groupoids \cite[\S1]{CS}.}
  This is why there is no canonical model of $\infty$-categories.

\subsection*{Reimagining the foundations of \texorpdfstring{$\infty$}{infinity}-category theory}

Despite these subtleties,  it is possible to reason rigorously and model-independently about $\infty$-categories without getting bogged down in the combinatorial scaffolding of a particular model. The framework introduced in \S\ref{sec:formal} considerably streamlines the basic core theory of $\infty$-categories, though its scope is currently  more limited than the corpus of results that have been proven using a model. 

However, this current state of the art employs proof techniques that are unfamiliar to non-category theorists  and thus is not feasible to integrate into the undergraduate curriculum. The concluding \S\ref{sec:synthetic} describes a more speculative dream for the future where enhancements to the foundations of mathematics would allow us to interpret uniqueness as contractibility  and automatically ensure that all constructions are invariant under equivalence.

\section{The formal theory of \texorpdfstring{$\infty$}{infinity}-categories}\label{sec:formal}

In the paper ``General theory of natural equivalences'' that marked the birth of category theory, Eilenberg and Mac Lane observed that categories, functors, and natural transformations assemble into a 2-category $\CAT$ \cite{EM}.    
 The essence of this result is the observation that natural transformations can be composed in 2-dimensions' worth of ways: ``vertically'' along a boundary functor or ``horizontally'' along a boundary category. 

More generally, Power proves that any \emph{pasting diagram} of compatibly oriented functors and natural transformations has a unique composite \cite{Power}. 
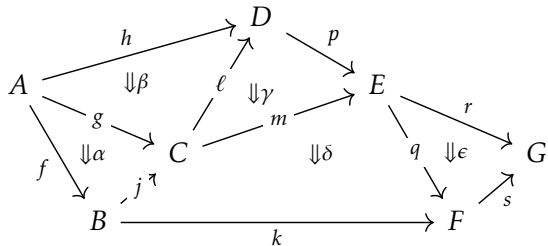
\[
\begin{tikzcd}[sep=small] & \arrow[ddr, phantom, "\scriptstyle\Downarrow\beta"]  & & D \arrow[drr, "p"] \\
A \arrow[ddr, "f"'] \arrow[drr, "g" description] \arrow[urrr, "h"] & & & & &  E \arrow[drr, "r"] \arrow[ddr, "q" description] & \arrow[dd, phantom, "\scriptstyle\Downarrow\epsilon"]  \\ \arrow[rr, phantom, "\scriptstyle\Downarrow\alpha"] 
& & C \arrow[uur, "\ell" description] \arrow[urrr, "m" description] \arrow[rrrr, phantom, "\scriptstyle\Downarrow\delta"] & \arrow[uu, phantom, "\scriptstyle\Downarrow\gamma"]   & &~ & ~ & G \\ & B \arrow[ur, "j" description] \arrow[rrrrr, "k"']  & & & & &  F \arrow[ur, "s"']
\end{tikzcd}
\]
The pasting composite, which in the example above defines a natural transformation between the categories $A$ and $G$ from the functor $rph$ to the functor $skf$, can be decomposed as a vertical composite of whiskerings of the atomic natural transformations $\alpha, \beta, \gamma,\delta,\epsilon$: for instance, the composite factors as $rp\beta$ followed by $r\gamma g$ followed by $\epsilon mg$ followed by $sqm\alpha$ followed by $s\delta f$, among 8 total possibilities. Power's theorem is that pasting composition is well-defined.

Similarly, in well-behaved models of $\infty$-categories:

\begin{prop}[Joyal, Riehl--Verity] $\infty$-categories, $\infty$-functors, and $\infty$-natural transformations assemble into a cartesian closed 2-category $\ooCAT$.
\end{prop}

This result was first observed in the quasi-category model by Joyal \cite{joyal-theory}. Since the category of quasi-categories is cartesian closed, it is enriched over itself, defining an $(\infty,2)$-category of quasi-categories. The 2-category of quasi-categories is obtained by a quotienting process, that maps each hom-quasi-category to its homotopy category, by applying the left adjoint to the (nerve) inclusion of categories into quasi-categories. The other ``well-behaved'' models---including complete Segal spaces, Segal categories, and 1-complicial sets among others---are also cartesian closed, so a similar construction defines a 2-category of complete Segal spaces, and so on.\footnote{While the meaning of the terms ``$\infty$-categories'' and ``$\infty$-functors'' in a given model is typically clear, the ``$\infty$-natural transformations,'' which can be understood as equivalence classes of 2-cells in the ambient $(\infty,2)$-category, are less evident.} To adopt a ``model-independent'' point of view, note that these 2-categories are all \emph{biequivalent} or ``the same'' in the sense appropriate to 2-category theory \cite[\S E.2]{RV}.

The good news, which is surprising to many experts, is that a fair portion of the basic theory of $\infty$-categories can be developed in the 2-category $\ooCAT$. These aspects might be described as \emph{formal category theory}, as they involve definitions of categorical notions such as equivalence, adjunction, limit, and colimit that can be defined internally to any 2-category. In the 2-category $\CAT$, these recover the classical notions from 1-category theory, while in the 2-category $\ooCAT$ these specialize to the correct notions in $\infty$-category theory.\footnote{The technical part of this story involves proofs that the ``synthetic'' or ``formal'' notions introduced here agree with the previously-defined ``analytic'' notions in the quasi-categories model  \cite[\S F]{RV}. But we encourage those not already acquainted with the analytic theory of some model of $\infty$-categories to learn these definitions first.} Thus, for the core basic theory involving these notions, ordinary category theory extends to $\infty$-categories simply by appending the prefix ``$\infty-$'' \cite[\S2-4]{RV}.

\subsection*{Equivalences and adjunctions}

The following definitions make sense in an arbitrary cartesian closed 2-category, such as  $\ooCAT$.

\begin{defn}\label{defn:equivalence} An \textbf{equivalence} between $\infty$-categories is given by:
    \begin{itemize}
        \item a pair of $\infty$-categories $A$ and $B$,
        \item a pair of $\infty$-functors $g \colon A \to B$ and $h \colon B \to A$, and
        \item a pair of invertible $\infty$-natural transformations $\alpha \colon \id_A \cong h g$, and $\beta \colon g h \cong  \id_B.$
  \end{itemize}
\end{defn}

\begin{defn}\label{defn:adjunction}
    An \textbf{adjunction} between $\infty$-categories is given by:
    \begin{itemize}
        \item a pair of $\infty$-categories $A$ and $B$,
        \item a pair of $\infty$-functors $u \colon A \to B$ and $f \colon B \to A$, and
        \item a pair of $\infty$-natural transformations $\eta \colon \id_B \To uf$ and $\epsilon \colon fu \To \id_A$
    \end{itemize}
    so that the following pasting identities hold:
    \[
        \begin{tikzcd}[column sep=1.5em]
            B \arrow[dr, "f"'] \arrow[rr, equals] & \arrow[d, phantom, "\scriptstyle\Downarrow\eta"] & B \arrow[dr, "f"] \arrow[d, phantom, "\scriptstyle\Downarrow\epsilon"] & & B  \arrow[d, bend left, "f"] \arrow[d, bend right, "f"'] \arrow[d, phantom, "="] & & B \arrow[d, phantom, "\scriptstyle\Downarrow\epsilon"] \arrow[rr, equals] \arrow[dr, "f" description] &\arrow[d, phantom, "\scriptstyle\Downarrow\eta"] & B \arrow[dr, phantom, "=\quad"] & B \\
            & A \arrow[ur, "u" description] \arrow[rr, equals] & ~ & A \arrow[ur, phantom, "=\quad"] & A & A \arrow[ur, "u"] \arrow[rr, equals] & ~ & A \arrow[ur, "u"'] & & A \arrow[u, bend left, "u"] \arrow[u, bend right, "u"'] \arrow[u, phantom, "="]
        \end{tikzcd} \, .
    \]
One commonly writes $\begin{tikzcd}
        A \arrow[r, bend right, "u"'] \arrow[r, phantom, "\bot"] & B \arrow[l, bend right, "f"']
    \end{tikzcd}$ 
    to indicate that $f$ is \textbf{left adjoint} to $u$, its \textbf{right adjoint}, with the data of the \textbf{unit} $\eta$ and \textbf{counit} $\epsilon$ being left implicit.
\end{defn}

Among the many advantages of using definitions that are taken ``off the shelf'' from the 2-categories literature is that the standard 2-categorical proofs then specialize to prove the following facts about adjunctions between $\infty$-categories.\footnote{What is less obvious is that these definitions are ``correct'' for $\infty$-category theory. In the case of equivalences, this is easily seen in any of the models. In the case of adjunctions, this is quite subtle, and involves the fact that the low-dimensional data enumerated in Definition \ref{defn:adjunction} suffices to determine a full ``homotopy coherent adjunction,'' uniquely up to a contractible space of choices \cite{RV-adj}.}

\begin{prop}[{\cite[2.1.9]{RV}}] Adjunctions between $\infty$-categories compose: 
    \[\begin{tikzcd}
        A \arrow[r, bend right, "u"'] \arrow[r, phantom, "\bot"] & B \arrow[l, bend right, "f"'] \arrow[r, bend right, "v"'] \arrow[r, phantom, "\bot"] & C \arrow[l, bend right, "g"'] & \rightsquigarrow & A \arrow[r, bend right, "v \circ u"'] \arrow[r, phantom, "\bot"] & C \arrow[l, "f \circ g"', bend right]
    \end{tikzcd}
    \]
\end{prop}

\begin{prop}[{\cite[2.1.10]{RV}}]
Given an adjunction  $\begin{tikzcd}
    A \arrow[r, bend right, "u"'] \arrow[r, phantom, "\bot"] & B \arrow[l, bend right, "f"']
\end{tikzcd}$ between $\infty$-cat\-e\-gories 
and a functor $f' \colon B \to A$,  $f'$ is left adjoint to $u$ if and only if there exists a natural isomorphism $f \cong f'$.
\end{prop}

\begin{prop}[{\cite[2.1.12]{RV}}]     Any equivalence can be promoted to an adjoint equivalence at the cost of replacing either of the natural isomorphisms.
\end{prop}

\subsection*{Relative adjunctions}

The unit and counit in an adjunction satisfy a universal property in the 2-category $\ooCAT$ that we now explore in the case of the counit $\epsilon \colon fu \To \id_A$. Given any $\infty$-functors $a \colon X \to A$ and $b \colon X \to B$, pasting with $\epsilon$ defines a bijection between $\infty$-natural transformations $\alpha \colon fb \To a$ and $\beta \colon b \To ua$. Any $\alpha \colon fa \To b$ factors through a unique $\beta \colon b \To ua$ as displayed below:
\begin{equation}\label{eq:counit-abs-lifting}
    \begin{tikzcd} X \arrow[r, "b"] \arrow[d, "a"'] \arrow[dr, phantom, "\scriptstyle\forall\Downarrow\alpha"] & B \arrow[d, "f"] & X \arrow[r, "b"] \arrow[d, "a"'] \arrow[dr, phantom, "\scriptstyle\exists!\!\Downarrow\!\beta" very near start, "\scriptstyle\Downarrow\epsilon" very near end] & B \arrow[d, "f"]\\ A \arrow[r, equals] & A \arrow[ur, phantom, "="] & A \arrow[r, equals] \arrow[ur, "u" description] & A
    \end{tikzcd}
\end{equation}
which is to say that the pair $(u,\epsilon)$ defines an \emph{absolute right lifting} of $\id_A$ through $f$. Indeed, $f$ is left adjoint to $u$ with counit $\epsilon$ if and only if $(u,\epsilon)$ defines an absolute right lifting of $\id_A$ through $f$ \cite[2.3.7]{RV}.

Any component $\epsilon a$ of the counit $\epsilon$ of an adjunction satisfies a universal property analogous to \eqref{eq:counit-abs-lifting}:
\[
\begin{tikzcd} Y \arrow[r, "b"] \arrow[d, "x"'] \arrow[dr, phantom, "\scriptstyle\forall\Downarrow\alpha"] & B \arrow[d, "f"] & Y \arrow[rr, "b"] \arrow[d, "x"'] \arrow[dr, phantom, "\scriptstyle\exists!\!\Downarrow\!\beta"]  & \arrow[dr, phantom, "\scriptstyle\Downarrow\epsilon" very near end] & B \arrow[d, "f"]\\ X \arrow[r, "a"'] & A \arrow[ur, phantom, "="] & X \arrow[r, "a"'] & A \arrow[r, equals] \arrow[ur, "u" description] & A
\end{tikzcd}
\]
which asserts that $(ua,\epsilon a)$ defines an absolute right lifting of $a$ through $f$. Motivated by examples such as these, absolute right liftings $(r,\rho)$ of a generic $\infty$-functor $g \colon C \to A$ through $f \colon B \to A$ can be thought of as exhibiting $r$ as \emph{ right adjoint to $f$ relative to $g$ with partial counit $\rho$}.

\begin{defn} Given $\infty$-functors $g \colon C \to A$ and $f \colon B \to A$ with common co\-do\-main, a functor $r \colon C \to B$ and $\infty$-natural transformation $\rho \colon fr \To g$ define an \textbf{absolute right lifting} of $g$ through $f$ or a \textbf{right adjoint of} $f$ \textbf{relative to} $g$ if pasting with $\rho$ induces a bijection between $\infty$-natural transformations as displayed below:
\[
    \begin{tikzcd} X \arrow[r, "b"] \arrow[d, "c"'] \arrow[dr, phantom, "\scriptstyle\forall\Downarrow\alpha"] & B \arrow[d, "f"] & X \arrow[r, "b"] \arrow[d, "c"'] \arrow[dr, phantom, "\scriptstyle\exists!\!\Downarrow\!\beta" very near start, "\scriptstyle\Downarrow\rho" very near end] & B \arrow[d, "f"]\\ C \arrow[r, "g"'] & A \arrow[ur, phantom, "="] & C \arrow[r, "g"'] \arrow[ur, "r" description] & A
    \end{tikzcd}
\]
\end{defn}

The following lemmas about relative right adjoints admit straightforward 2-categorical proofs:

\begin{lem}[{\cite[2.3.6]{RV}}]\label{lem:absolute-restriction} If $(r,\rho)$ is right adjoint to $f \colon B \to A$ relative to $g \colon C \to A$ then for any $c \colon X \to C$, $(rc,\rho c)$ is right adjoint to $f$ relative to $gc$.
    \[
        \begin{tikzcd}&  \arrow[dr, phantom, "\scriptstyle\Downarrow\rho" very near end] & B\arrow[d, "f"] \\ X \arrow[r, "c"'] & C \arrow[r, "g"'] \arrow[ur, "r"] & A
        \end{tikzcd}
    \]
\end{lem}

\begin{lem}[{\cite[2.4.1]{RV}}]\label{lem:absolute-pasting}
Suppose that $(r,\rho)$ is right adjoint to $f \colon B \to A$ relative to $g \colon C \to A$. Then in any diagram of $\infty$-functors and $\infty$-natural transformations
\[
    \begin{tikzcd} \arrow[dr, phantom, "\scriptstyle\Downarrow\sigma" very near end] & D \arrow[d, "k"] \\ \arrow[dr, phantom, "\scriptstyle\Downarrow\rho" very near end] & B \arrow[d, "f"] \\ C \arrow[r, "g"'] \arrow[ur, "r"] \arrow[uur, "s", bend left] & A
    \end{tikzcd}
\]
$(s,\sigma)$ is right adjoint to $k$ relative to $r$ if and only if $(s, \rho \cdot f\sigma)$ is right adjoint to $fk$ relative to $g$.
\end{lem}

\subsection*{Limits and colimits}

Relative adjunctions can be used to define limits and colimits of diagrams valued inside an $\infty$-category $A$ and indexed by another $\infty$-category $J$. As a cartesian closed 2-category,  $\ooCAT$ contains a terminal $\infty$-category $1$, which admits a unique $\infty$-functor $! \colon J \to 1$ from any other $\infty$-category $J$. Maps $a \colon 1 \to A$ from the terminal $\infty$-category into another $\infty$-category define \emph{elements} in the $\infty$-category $A$.\footnote{We find this terminology less confusing than referring to the ``objects'' in an $\infty$-category $A$, which is itself an object of $\ooCAT$.} Using the cartesian closed structure, any diagram $d \colon J \to A$ defines an element $d \colon 1 \to A^J$ in the $\infty$-\emph{category of} $J$-\emph{shaped diagrams in} $A$. Exponentiation with the unique functor $! \colon J \to 1$ defines the \emph{constant diagrams functor} $\Delta \colon A \to A^J$, which carries an element of $A$ to the constant $J$-shaped diagram at that element.

\begin{defn}\label{defn:limit-colimit} An $\infty$-category $A$ \textbf{admits $J$-shaped limits} if the constant diagrams functor $\Delta \colon A \to A^J$ admits a right adjoint and \textbf{admits $J$-shaped colimits} if the constant diagrams functor admits a left adjoint:
    \[
        \begin{tikzcd} A \arrow[r, "\Delta" description] & A^J \arrow[l, bend left=40, "\lim", "\bot"'] \arrow[l, bend right=40, "\colim"' pos=.45, "\bot" pos=.6]
        \end{tikzcd}
    \]
\end{defn}

The counit $\lambda$ of the adjunction $\Delta\dashv\lim$ and the unit $\gamma$ of the adjunction $\colim\dashv\Delta$ encode the limit and colimit cones, as is most easily seen when considering their components at a diagram $d \colon 1 \to A^J$. By Lemma \ref{lem:absolute-restriction}, the right and left adjoints restrict to define relative adjunctions:
\[
    \begin{tikzcd}
        & \arrow[dr, phantom, "\scriptstyle\Downarrow\lambda" very near end] & A \arrow[d, "\Delta"] & & \arrow[dr, phantom, "\scriptstyle\Uparrow\gamma" very near end] & A \arrow[d, "\Delta"] \\ 1 \arrow[r, "d"'] & A^J \arrow[ur, "\lim"] \arrow[r, equals] & A^J & 1 \arrow[r, "d"'] & A^J \arrow[ur, "\colim"] \arrow[r, equals] & A^J
    \end{tikzcd}
\]
These observations allow us to generalize Definition \ref{defn:limit-colimit} to express the universal properties of the limit or colimit of a single diagram in settings where the limit and colimit functors do not exist.

\begin{defn}\label{defn:elementwise-limit-colimit} A diagram $d \colon J \to A$ between $\infty$-categories \textbf{admits a limit} just when $\Delta$ admits a relative right adjoint at $d$ and \textbf{admits a colimit} just when $\Delta$ admits a relative left adjoint at $d$, as encoded by absolute lifting diagrams:
    \[
        \begin{tikzcd}
        \arrow[dr, phantom, "\scriptstyle\Downarrow\lambda_d" very near end] & A \arrow[d, "\Delta"]  & \arrow[dr, phantom, "\scriptstyle\Uparrow\gamma_d" very near end] & A \arrow[d, "\Delta"] \\ 1 \arrow[r, "d"']  \arrow[ur, "\lim d"] & A^J & 1 \arrow[r, "d"']  \arrow[ur, "\colim d"]  & A^J
        \end{tikzcd}
    \]
\end{defn}

There is an easy formal proof of the $\infty$-categorical version of a classical theorem:

\begin{thm}
    Right adjoints preserve limits and left adjoints preserve colimits.
\end{thm}
\begin{proof}
    Consider an adjunction as in Definition \ref{defn:adjunction} and a limit of a diagram as in Definition \ref{defn:elementwise-limit-colimit}. To show that the limit is preserved by the right adjoint, we must show that
    \[
        \begin{tikzcd}
        \arrow[dr, phantom, "\scriptstyle\Downarrow\lambda_d" very near end] & A \arrow[d, "\Delta"]  \arrow[r, "u"] & B \arrow[d, "\Delta"] \\ 1 \arrow[r, "d"']  \arrow[ur, "\lim d"] & A^J \arrow[r, "u^J"'] & B^J
        \end{tikzcd}
    \]
    defines a relative right adjoint. By Lemma \ref{lem:absolute-pasting}, it suffices to demonstrate this after pasting with the relative right adjoint encoded by the component at $d$ of the counit $\epsilon^J$ of the adjunction $f^J \dashv u^J$, as below-left:
    \[
\begin{tikzcd}
\arrow[dr, phantom, "\scriptstyle\Downarrow\lambda_d" pos=.85] & A \arrow[d, "\Delta"]\arrow[r, "u"] & B \arrow[d, "\Delta"] & [-3pt]{~} &\arrow[dr, phantom, "\scriptstyle\Downarrow\epsilon" pos=.85] & B \arrow[d, "f"]  & \arrow[dr, phantom, "\scriptstyle\Downarrow\epsilon_{\lim d}" pos=.8] & B \arrow[d, "f"] \\  1 \arrow[ur, "\lim d"] \arrow[r, "d"'] & A^J \arrow[dr, equals] \arrow[r, "{u^J}"] & B^J \arrow[d, "{f^J}"] \arrow[dl, phantom, "{\scriptstyle\Downarrow\epsilon^J}" pos=.15] \arrow[r, phantom, "="] & ~ \arrow[dr, phantom, "\scriptstyle\Downarrow\lambda_d" pos=.85]  & A \arrow[ur, "u"] \arrow[r, equals] \arrow[d, "\Delta"] & A \arrow[d, "\Delta"]
  \arrow[r, phantom, "="] & ~ \arrow[dr, phantom, "\scriptstyle\Downarrow\lambda_d" pos=.9] & A \arrow[d, "\Delta"] \\ & ~ & A^J & 1 \arrow[r, "d"'] \arrow[ur, "\lim d"] & A^J \arrow[r, equals] & A^J & 1 \arrow[r, "d"'] \arrow[ur, "\lim d" {xshift=2pt}] \arrow[uur, bend left, "u\lim d"] & A^J
\end{tikzcd}
\]
From the 2-functoriality of the exponential in a cartesian closed 2-category,
$f^J\Delta = \Delta f$ and $\epsilon^J \Delta = \Delta \epsilon$. Hence, the pasting diagram displayed above-left equals the one displayed above-center, which equals the diagram above-right. This latter diagram is a pasted composite of relative right adjoints, and is then a relative right adjoint in its own right by Lemma \ref{lem:absolute-pasting}.
\end{proof}

\subsection*{Universal properties}

Definitions \ref{defn:adjunction} and \ref{defn:elementwise-limit-colimit} do not express the full universal properties of adjunctions between and co/limits in $\infty$-categories. These require additional structure borne by the 2-category $\ooCAT$, namely the existence of \emph{comma $\infty$-categories} for any pair of $\infty$-functors $f \colon B \to A$ and $g \colon C \to A$. The comma $\infty$-category is an $\infty$-category $\Hom_A(f,g)$ equipped with canonical functors and an $\infty$-natural transformation as below-left, formed by the pullback of $\infty$-categories below-right:

\begin{equation}\label{eq:comma-cat}
    \begin{tikzcd}
        \Hom_A(f,g) \arrow[r, "\dom"] \arrow[d, "\cod"'] \arrow[dr, phantom, "\scriptstyle\Downarrow\phi"] & B \arrow[d, "f"]  & & \Hom_A(f,g) \arrow[dr, phantom, "\lrcorner" very near start]\arrow[r, "\phi"] \arrow[d, "{(\cod,\dom)}"'] & A^\2 \arrow[d, "{(\cod,\dom)}"] \\ C \arrow[r, "g"'] & A & & C \times B \arrow[r, "g \times f"'] & A \times A
    \end{tikzcd}
\end{equation}
While the universal property of $\Hom_A(f,g)$ as an object of $\ooCAT$ is weaker than the standard notion of comma object in 2-category theory, nevertheless:

\begin{thm}[{\cite[3.5.8]{RV}}]\label{thm:commas} An $\infty$-functor $f \colon B \to A$ admits a right adjoint $r \colon C \to B$  relative to $g \colon C \to A$ if and only if there exists an equivalence over $C \times B$:
    \[ \Hom_A(f,g) \simeq_{C \times B} \Hom_B(B,r). \]
\end{thm}

When Theorem \ref{thm:commas} is applied to $f \colon B \to A$ and $\id_A$, it specializes to:

\begin{cor}[{\cite[4.1.1]{RV}}] An $\infty$-functor $f \colon B \to A$ admits a right adjoint $u \colon A \to B$ if and only if $\Hom_A(f,A) \simeq_{A \times B} \Hom_B(B,u)$.
\end{cor}

When Theorem \ref{thm:commas} is applied to $\Delta \colon A \to A^J$ and $d \colon 1 \to A^J$ it specializes to:

\begin{cor}[{\cite[4.3.1]{RV}}] A diagram $d \colon J \to A$ has a limit $\ell \colon 1 \to A$ if and only if $\Hom_{A^J}(\Delta,d) \simeq_A \Hom_A(A,\ell)$.
\end{cor}

Here the comma $\infty$-category $\Hom_{A^J}(\Delta,d)$ defines the \emph{$\infty$-category of cones} over $d$. The $\infty$-category $\Hom_A(A,\ell)$ admits a \emph{terminal element} $\id_\ell \colon 1 \to \Hom_A(A,\ell)$,  which defines a right adjoint to the unique functor $! \colon \Hom_A(A,\ell) \to 1$. 
Via the equivalence $\Hom_{A^J}(\Delta,d) \simeq_A \Hom_A(A,\ell)$ we see that the limit cone is terminal in the $\infty$-category of cones. Indeed, $d$ admits a limit if and only if the $\infty$-category of cones admits a terminal element \cite[4.3.2]{RV}.

\subsection*{Could we teach this to undergraduates?} If classical category theory were a standard course in the undergraduate mathematics curriculum, then perhaps students with a particular interest in the subject might go on to learn some 2-category theory, like students with a particular interest in algebra might go on to take more specialized courses in classical algebraic geometry, algebraic number theory, or representation theory when offered. From this point, the proofs of the results stated here are not too difficult.

But these results encompass only a small portion of $\infty$-category theory, and for most of the rest there is a greater jump in complexity when extending from 1-categories to $\infty$-categories. This is best illustrated by considering the  Yoneda lemma, which in an important special case is stated as follows:

\begin{thm}[{\cite[5.7.1]{RV}}]\label{thm:yoneda} Given an $\infty$-category $A$ and elements $a,b \colon 1 \to A$, the $\infty$-groupoid $\Hom_A(a,b)$ is equivalent to the $\infty$-groupoid of functors $\Hom_A(A,a) \to \Hom_A(A,b)$ over $A$.
\end{thm}

One direction of this equivalence is easy to describe: the map from right to left is defined by evaluation at the identity element $\id_a \colon 1 \to \Hom_A(a,a)$. The inverse equivalence defines the Yoneda embedding, which is notoriously difficult to construct in $\infty$-category theory as it involves equipping $A$ with a homotopy coherent composition function \cite[\S 5.1.3]{lurie-topos}.

To achieve further simplifications of $\infty$-category theory, one idea is to ask our foundation system to do more of the work.

\section{A synthetic theory of \texorpdfstring{$\infty$}{infinity}-categories}\label{sec:synthetic}

To explain the desiderata for an alternative foundation, consider the default notions of ``sameness'' for $\infty$-categorical data. Two elements in an $\infty$-category are the same if and only if they are connected by a path in the underlying $\infty$-groupoid, while parallel morphisms are the same if and only if they are connected by an invertible 2-cell, defining a path in the appropriate mapping space. For an $\infty$-categorical construction to be ``well-defined'' it must:
\begin{enumerate}
    \item respect the notions of sameness encoded by paths in suitable spaces, and 
    \item respect equivalences between these spaces themselves, since they are only well-defined as homotopy types.
\end{enumerate} 

\subsection*{Axiomatizing sameness}

In traditional foundations, there is a similar axiom that mathematical constructions or results must respect sameness as encoded by equality. The axioms that define the binary relation ``$=$'' in first-order logic are:
\begin{itemize}
    \item reflexivity: for all $x$, $x=x$ is true.
    \item indiscernibility of identicals: for all $x,y$ and for all predicates $P$, if $x=y$ then $P(x)$ holds if and only if $P(y)$ holds \cite{leibniz}. 
\end{itemize} 

Analogous rules govern Martin-L\"{o}f's \emph{identity types} in an alternative foundational framework for constructive mathematics known as \emph{dependent type theory} \cite{PML}. Here the primitives include \emph{types} like $\NN$, $GL_n(\RR)$, and $\Group$ and \emph{terms} like $17 : \NN$, $I^n : GL_n(\RR)$, $S_n : \Group$, which may depend on an arbitrary \emph{context} of variables drawn from previously defined types (e.g., the $n :\NN$ appearing in three examples above).\footnote{In a mathematical statement of the form ``Let \ldots be  \ldots then \ldots" the stuff following the ``let'' likely declares the names of the variables in the context described after the ``be'', while the stuff after the ``then'' most likely describes a type or term in that context.} In traditional foundations, mathematical structures are defined in the language of set theory, while proofs obey the rules of first-order logic. In dependent type theory, mathematical constructions and mathematical proofs are unified: both are given by terms in appropriate types. The ambient type then describes the sort of object being constructed or the mathematical statement being proven, while the context describes the inputs to the construction or the hypotheses for the theorem.

The rules for Martin-L\"{o}f's identity types include:\footnote{One additional rule, not listed here, concerns the computational behavior of this system.}
\begin{itemize}
    \item identity-formation: in the context of two variables $x,y :A$ there is a type $x=_Ay$. 
\end{itemize}
Types encode mathematically meaningful statements or assertions, so this rule says that it is reasonable to inquire whether $x$ and $y$ might be identified, once $x$ and $y$ are terms in the same type. Certain identity types, such as $3=_\NN 4$, will be empty, since the terms $3$ and $4$ cannot be identified, but the type nevertheless exists.
\begin{itemize}
    \item identity-introduction: for any $x : A$, there is a term $\refl_x : x=_Ax$. 
\end{itemize}
Terms in types are witnesses to the truth of the statement encoded by the type, so this rule corresponds to the reflexivity axiom.
\begin{itemize}
    \item identity-elimination: given any family of types $Q(x,y,p)$ depending on terms $x, y : A$ and $p : x=_Ay$, to provide a family of terms $q_{x,y,p} : Q(x,y,p)$ for all $x,y,p$ it suffices to provide a family of terms $d_x : Q(x,x,\refl_x)$ for all $x$. 
\end{itemize}
This rule implies Leibniz's indiscernibility of identicals and much more besides:

\begin{thm}[{\cite{vBG,lumsdaine}}] The iterated identity types provide any type with the structure of an $\infty$-groupoid, in which the terms are the points, the identifications are the paths, and the higher identifications encode homotopies.
\end{thm}

In low dimensions, the $\infty$-groupoid structure on the iterated identity types of a type $A$ provides:
\begin{itemize}
    \item functions $ (-)^{-1} : (x=_A y) \to (y=_Ax)$ that invert identifications,
    \item functions $-\ast- :  (x=_Ay) \times (y=_Az) \to (x=_A z)$ that compose identifications,
    \item higher identifications $\textup{assoc} : (p \ast q) \ast r =_{x =_A w} p \ast (q \ast r)$ between composable triples of identifications
\end{itemize}
and much more. All of these terms are constructed using the identity-elimination rule from the reflexivity terms provided by the identity-introduction rule. Since an identification $p : x=_A y$ defines a path from $x$ to $y$ in the $\infty$-groupoid $A$, identifications in dependent type theory are often referred to colloquially as ``paths.''

\subsection*{The homotopical interpretation of dependent type theory}

The homotopical interpretation of dependent type theory was discovered in the early 21st century by Awodey--Warren \cite{AW}
 and Voevodsky \cite{V},   
  building on earlier work of Hofmann and Streicher  \cite{HS}. 
   This connection inspired Voevodsky to make the following definition:

\begin{defn}\label{defn:contractible} A type $A$ is \textbf{contractible} just when there is a term of type \[ \Sigma_{x :A} \Pi_{y : A} x=_A y\]
    which may be read as ``there exists $x:A$, so that for all $y :A$, $x$ is identifiable with $y$'' with all quantifiers interpreted as requiring continuous specifications of data.\footnote{The definition of contractible types makes use of the \emph{dependent pair} and \emph{dependent function} types in dependent type theory. A term of type $\Sigma_{x :A} \Pi_{y : A} x=_A y$ provides a term $c : A$, the ``center of contraction,'' together with a family of paths $p_z : c =_A z$ for all $z : A$, the ``contracting homotopy.''}
\end{defn}

Just as contractible spaces might contain uncountably infinitely many points, contractible types might contain more than one term. But since any two terms in a contractible type may be identified, and since identity-elimination implies that all of dependent type theory respects identifications, the theory behaves as if a contractible type had a unique term.

\subsection*{Univalent foundations}

Martin-L\"{o}f's dependent type theory is a formal system in which all constructions are continuous in paths, our first desideratum. The second desideratum, of well-definedness under equivalences between types, follows once Voevodsky's \emph{univalence axiom} is added to dependent type theory, resulting in a formal system called \emph{homotopy type theory}  or \emph{univalent foundations} \cite{rijke, HoTT}.

Voevodsky's univalence connects two notions of sameness for types. In traditional dependent type theory, types can be encoded as terms in a universe of types.\footnote{Russell's paradox can be resolved by a cumulative hierarchy of universes.} Thus, one notion of sameness between types $A, B$ belonging to a universe $\cU$ is given by paths $p : A=_\cU B$ in the universe. Another notion of sameness is suggested by types as $\infty$-groupoids: for any $A$ and $B$ there is a type $A \simeq B$ of (homotopy) equivalences from $A$ to $B$.\footnote{It is an interesting challenge to define a type whose data witnesses that a given map $f \colon A \to B$ is an equivalence in such a way that this type is contractible whenever it is inhabited \cite[\S4]{HoTT}.}
 By identity-elimination, to define a natural map $\textup{id-to-equiv}: A =_\cU B \to A \simeq B$ it suffices to define the image of $\refl_A : A =_\cU A$, which we take to be the identity equivalence. The \emph{univalence axiom} asserts that this map is an equivalence for all types $A$ and $B$.

By univalence, an equivalence $e : A \simeq B$ gives rise to a path $\ua(e) : A =_\cU B$, which can then be used to transport terms in any type involving $A$ to terms in the corresponding type involving $B$. Thus, univalent mathematics is automatically invariant under equivalence between types.

\subsection*{\texorpdfstring{$\infty$}{Infinity}-categories in univalent foundations}

There is an experimental exploration  of $\infty$-category theory \cite{RS,BuW,CBM} in an extension of homotopy type theory in which every type $A$  has a family of arrows $\Hom_A(x,y)$  in addition to the family of paths $x=_Ay$. These types are obtained from a new type-forming operation that produces  \emph{extension types}, whose terms are type-valued diagrams that strictly extend a given diagram along an inclusion of ``shapes'' (polytopes embedded in directed
cubes constructed in the theory of a strict interval) \cite[\S2-4]{RS}. Extension types include types analogous to the pullbacks of \eqref{eq:comp-space} and \eqref{eq:comma-cat}.

The formal system has semantics in complete Segal spaces \cite{RS,Weinberger} so it provides a rigorous way to prove theorems about $\infty$-categories as understood in traditional foundations. At the same time, the experience of working with $\infty$-categories in this ``univalent'' foundational setting is much more akin to the experience of working with 1-categories in traditional foundations.

It takes work to set up this formal system and considerably more to describe its interpretation in complete Segal spaces; indeed, this is where the hard work of solving homotopy coherence problems goes. But once this is done,  it is possible to define the notion of an $\infty$-category, something that would no doubt be reassuring to undergraduates hoping to learn about them.\footnote{Recall that in the formal $\infty$-category theory developed in the 2-category $\ooCAT$, a concrete definition of $\infty$-categories is not used.} The two axioms express the ``Segal'' and ``completeness'' conditions of Rezk's model \cite{rezk-CSS}, respectively. 

\begin{defn}[{\cite[5.3,10.6]{RS}}] An $\infty$-\textbf{category} is a type $A$ in which:
    \begin{itemize}
        \item every composable pair of arrows has a unique composite, and
        \item for any pair of terms $x, y : A$ the natural map $x=_A y \to x \cong_A y$ from paths in $A$ to isomorphisms in $A$ is an equivalence.
    \end{itemize}
\end{defn}

The ``uniqueness'' in the first axiom is in the sense of Definition \ref{defn:contractible}: what it asserts is that a suitable space of composites analogous to Definition \ref{defn:space-of-composites} is contractible---the default meaning of uniqueness in this formal system. Like in the space defined by \eqref{eq:comp-space}, a term in the space of composites provides a higher-dimensional witness to the composition relation. A composition function \[\circ \colon \Hom_A(y,z) \times \Hom_A(x,y) \to \Hom_A(x,z),\] is obtained by throwing away these witnesses. Like any construction in homotopy type theory, this function respects identifications and is therefore well-defined. It follows from this axiom that composition is associative and unital, using canonical identity arrows $\id_x : \Hom_A(x,x)$.

The second axiom involves a type of isomorphisms in $A$, which may be defined in a similar manner to the type of equivalences between types \cite[\S 10]{RS}.

\subsection*{The Yoneda lemma}

The advantages of this synthetic framework for $\infty$-category theory are on display when comparing the proofs of the Yoneda lemma in \cite[5.7.1]{RV} and in \cite[9.1]{RS}. We sketch the latter here in the same special case considered in Theorem \ref{thm:yoneda}.

\begin{thm}[{\cite[9.1]{RS}}]\label{thm:yoneda-again} 
    Given an $\infty$-category $A$ and elements $a,b \colon 1 \to A$, the type $\Hom_A(a,b)$ is equivalent to the type $\Pi_{x :A} \Hom_A(x,a) \to \Hom_A(x,b)$ of fiberwise functions.
\end{thm}

A key difference between Theorem \ref{thm:yoneda} and Theorem \ref{thm:yoneda-again} is that in the present framework it is straightforward to define the inverse equivalence to the evaluation at identity map. The inverse equivalence takes an arrow $f: \Hom_A(a,b)$ to the natural map $x \mapsto g \mapsto f \circ g : \Pi_{x:A} \Hom_A(x,a) \to \Hom_A(x,b)$. The usual proof then demonstrates that these maps are inverse equivalences. In one direction, this is given by an identification $f \circ \id_a =_{\Hom_A(a,b)} f$, which expresses the fact that composition in an $\infty$-category is unital. In the other direction, we must show that a fiberwise map $\phi$ agrees with the map whose component at $x$ is given by $g \mapsto  \phi_a(\id_a) \circ g: \Hom_A(x,a) \to \Hom_A(x,b)$. By a consequence of the univalence axiom called ``function extensionality,'' it suffices to show that $\phi_a(\id_a) \circ g =_{\Hom_A(x,b)} \phi_x(g)$. This follows from the fact that fiberwise maps are automatically ``natural,'' providing an identification $\phi_a(\id_a) \circ g =_{\Hom_A(x,b)} \phi_x(\id_a \circ g)$. Since the function $\phi_x$ respects the identification $\id_a \circ g =_{\Hom_A(x,a)} g$, we obtain the required identification.

In fact, \cite[9.5]{RS} proves a generalization of the Yoneda lemma  that was first discovered with this formal system and later proven in traditional foundations \cite[5.7.2]{RV}

\subsection*{Conclusion}

Significant technical problems remain to make it feasible to teach  $\infty$-category theory to undergraduates, but I have hope that in the future the subject  will not seem as forbidding as it does today.

\end{document}